\documentclass[smallextended,envcountsect]{svjour3}
\usepackage{pifont}
\usepackage{booktabs}
\usepackage{epsfig,color,graphicx,subfigure,amsmath,amssymb,multirow,lscape}
\usepackage[colorlinks,linkcolor=blue,citecolor=blue]{hyperref}
\usepackage{tikz,cite}
\usepackage{algorithm,algorithmic}
\usepackage{epstopdf}
\graphicspath{{pictures/}}
\DeclareGraphicsExtensions{.pdf,.jpeg,.png}
\usepackage[figuresright]{rotating}
\usetikzlibrary{arrows, automata}
\smartqed
\journalname{}

\newtheorem{thm}{Theorem}[section]
\newtheorem{lem}{Lemma}[section]

\newtheorem{rem}{Remark}[section]

\def\be{\begin{eqnarray}}
\def\ee{\end{eqnarray}}
\def\ben{\begin{eqnarray*}}
\def\een{\end{eqnarray*}}
\def\ba{\begin{array}}
\def\ea{\end{array}}
\def\bi{\begin{itemize}}
\def\ei{\end{itemize}}

\def\null{{\rm null}}

\newcommand{\R}{\mathbb R}

\begin{document}

\title{A splitting-based KPIK method for eddy current optimal control problems in an  all-at-once approach\thanks{The work is supported by the National Natural Science Foundation of China (Grant Nos. 12126344, 12126337 and 11901324) and the China Scholarship Council (No. 202308350044).
}}
\titlerunning{A splitting-based KPIK method}
\author{Min-Li Zeng $^{1,2,*}$ \and Martin Stoll$^{2}$ \thanks{$ ^*$ Corresponding author: zengminli@ptu.edu.cn, zenm@hrz.tu-chemnitz.de}}
\institute{1 \ Fujian Key Laboratory of Financial Information Processing, Putian University, Putian, 351100, China.\\
 2 \ Faculty of Mathematics, Technische Universit$\ddot{a}$t Chemnitz, Chemnnitz, D-09107, Germany. }
\date{ }

\maketitle

\begin{abstract}
In this paper, we focus on efficient methods to fast solve discretized linear systems obtained from eddy current optimal control problems in an all-at-once approach. We construct a new low-rank matrix equation method based on a special splitting of  the coefficient matrix and the Krylov-plus-inverted-Krylov (KPIK) algorithm.  Firstly, we rewrite the resulting discretized linear system in a matrix-equation form. Then using the  KPIK algorithm,  we can obtain the low-rank approximation solution. The new method is named the splitting-based Krylov-plus-inverted-Krylov (SKPIK) method. The SKPIK method can not only solve the large and sparse discretized systems fast but also overcomes the storage problem. Theoretical results about the existence of the low-rank solutions are given. Numerical experiments are used to illustrate the performance of the new low-rank matrix equation method by compared with some existing classical efficient methods.
\end{abstract}



\section{Introduction}\label{sec:intro}

In this work, we are interested in the numerical solution of the distributed optimal control problem: Find the state $y$ and the control $u$ that minimizes the cost functional
\be\label{eddy-eq:1}
J(y,u)=\frac{1}{2}\int_0^T\int_{\Omega}|y-y_d|^2dxdt+\frac{\beta}{2}\int_0^T\int_{\Omega}|u|^2dxdt,
\ee
subject to
\be\label{eddy-eq:2}
\left\{
\begin{array}{rlcl}
\sigma\frac{\partial}{\partial t}y+\textbf{curl}(\nu \textbf{curl}y)&=u, &\text{in}&\Omega\times (0,T),\\
y\times n&=0,& \text{on}& \partial\Omega\times (0,T),\\
y&=y_0,& \text{on}& \Omega\times \{0\},\\
\end{array}
\right.
\ee
where $y_d$ is the desired state, $\sigma\in L^{\infty}(\Omega)$ denotes the conductivity, $\nu\in L^{\infty}(\Omega)$ is the reluctivity parameter  and $\beta>0$ is a cost parameter. We assume that $\Omega\in \mathbb{R}^2$ or $\mathbb{R}^3$ is a bounded Lipschitz domain.

The solution of eddy current optimal control problems is one of the most interesting and demanding problems in applied mathematics and scientific computing. Eddy current optimal control problems are modeled by Maxwell's equations \cite{kolmbauer2013efficient,kolmbauer2011frequency}. Typically, this category of issues encompasses a set of interrelated variable functions, such as the electromagnetic force, electric current, and an associated adjoint variable function. Over the last decades, the numerical solution of the eddy current optimal control problems has received lots of attention \cite{axelsson2019note,axelsson2017preconditioners,axelsson2019preconditioning,bachinger2006efficient}.

When the eddy current optimal control problem is  time-harmonic$\setminus$time-periodic, Kolmbauer and Langer used the multi-harmonic finite element method (FEM) to discretize the problem in \cite{kolmbauer2011frequency,kolmbauer2012robust}  and then constructed a robust preconditioned minimum residual (MINRES) solver for the corresponding disretized linear system. In contrast to previous approaches, 
Kolmbauer further derived the problem setting in a mixed variational formulation under the cases of different control and observation domains, observation at a final time, constraints to the control and the state and observation in certain energy spaces. More details can be found in  \cite{kolmbauer2013robust,kolmbauer2013efficient,bachinger2006efficient}. During the past two decades, more and more research concerning the discretization systems and the corresponding preconditioned methods has been carried out. For example, in \cite{wolfmayr2023posteriori}, Wolfmayr  presents the multi-harmonic analysis and derivation of functional type a posteriori estimates of a distributed eddy current optimal control problem. Axelsson and co-authors presented the efficient preconditioners for the time-harmonic optimal control eddy-current problems in \cite{axelsson2017preconditioners} and for eddy-current optimally controlled time-harmonic electromagnetic problem in \cite{axelsson2019preconditioning}, respectively. We refer to \cite{buffa2000justification,axelsson2019note,zeng2021respectively} and references therein  for more efficient preconditioning techniques for the time-harmonic$\setminus$time-periodic eddy current optimal control problems.

When the eddy current problems are not time-harmonic$\setminus$time-periodic, then the first-order temporal derivative is  complicated. Because the classical time-stepping method, i.e., solving the partial differential equations (PDEs) one-time step after one-time step, would be time-consuming if the number of time steps is large. Therefore, the so-called all-at-once approach technique \cite{hinze2008hierarchical,hinze2012space,benzi2011preconditioning,stoll2013all} becomes popular. The all-at-once approach involves discretizing the relevant issue within the space-time continuum and resolving it for all temporal intervals simultaneously. A key benefit of this method is the deferment of meeting the optimal control problem criteria until the entire system converges. Nonetheless, a significant drawback is the resultant large-scale systems inherent to the all-at-once approach. Hence, developing efficient algorithms for  fast solving large-scale systems has always been a research focus.

The first motivation for fast solving large-scale systems is the development of parallel-in-time (PinT) methods for
evolutionary PDEs \cite{lions2001resolution}.  Among these, we mention the parareal algorithm \cite{maday2002parareal} and a closely related algorithm multigrid-reduction-in-time (MGRiT) algorithm \cite{falgout2014parallel}, which attracted considerable attention in recent years. 
The convergence properties of the parareal algorithm and MGRiT are discussed in \cite{maday2002parareal,gander2007analysis}. Many
efforts are devoted to improving these two PinT algorithms, and in particular the
authors \cite{wu2018toward,wu2019acceleration} proposed a novel coarse grid correction, which shows great
potential for increasing the speedup according to the numerical results in \cite{kwok2019schwarz}. There
are also many other PinT algorithms with completely different mechanisms from the
parareal algorithm and MGRiT, such as the space-time multigrid algorithms \cite{gander2016analysis,horton1995space} and the diagonalization-based all-at-once algorithms \cite{gander2016direct,mcdonald2018preconditioning,goddard2018note}. For an
overview, we refer the interested reader to \cite{gander201550}. However, in contrast to non-PinT techniques, a limitation of PinT methods is their substantial demand for experimental resources, such as a large array of processors, leading to increased experimental expenses. Additionally, PinT methods continue to face storage challenges due to the demands of space-time discretization.

To solve large-scale systems based on an all-at-once approach efficiently, people usually consider using iteration methods, for example, Krylov subspace methods combined with efficient preconditioners. Much progress has been made in recent years concerning  efficient preconditioners for the all-at-once-based discretized systems for different applications. In  \cite{rees2010all}, Rees and co-authors  discussed implementation details on the all-at-once approach with different boundary conditions. McDonald, Pestana, and Wathen established a block circulant preconditioner for the all-at-once evolutionary PDE system  in \cite{mcdonald2018preconditioning}. More applications can be found in \cite{haber2001preconditioned,zhao2021preconditioning} and references therein.

However, it is easy to see that the models based on the all-at-once approach result in millions of variables. Therefore,  it needs huge memory to store the large and dense solution computed by standard algorithms.  Nevertheless, neither PinT methods nor the previously described techniques can overcome the existing memory limitations. Ultimately, due to their size and complexity, solving such huge models is still a challenging task.

Therefore, efficient solvers based on low memory requirements are desirable. Recently, B\"{u}nger and co-authors proposed a new algorithm framework \cite{bunger2021low}, which can be used to solve PDE-constrained optimization problems efficiently.  To overcome the obstacles brought out by the storage requirement of vast dimensionality, one can  reformulate the  Karush-Kuhn-Tucker (KKT) system  into a Sylvester-like matrix equation \cite{stoll2015low}, which would provide an efficient representation with a minimal amount of storage.  Then based on this reformulation, we can search for the low-rank approximation solution by using the idea of low-rank in time technique,  which can be found in \cite{markovsky2012low}.  The main idea of the low-rank solver is to compute a projection of the solution onto a small subspace. Hence, the low-rank solver is similar to model order reduction (MOR) approaches. However, in contrast to classical MOR schemes,  the low-rank solvers do not need to compress the full solution at the end of the algorithm but start with low-rank data and maintain this form throughout the iteration. Hence, more and more attention has been paid to the low-rank solvers. For example, the low-rank approximation solution algorithm has been used for solving the optimal control problem constrained with a forward Navier-Stokes equation by Dolgov in \cite{dolgov2017low}, the optimal control problem constrained with unsteady Stokes-Brinkman equations by  Benner and co-authors  in \cite{benner2015low} and  two-dimensional time-dependent Navier-Stokes equations with a stochastic by Benner and co-authors in \cite{benner2016low}. More low-rank approximation algorithms for optimization of large-scale systems have been studied in existing papers, and we refer the readers to \cite{benner2013low,benner2016block,cichocki2017tensor,dolgov2016fast,dolgov2017low,benner2023low} for recent accounts.

In this work, we  focus on the efficient algorithms based on the low-rank matrix equation method for solving the eddy current optimal control problems by an all-at-once approach.  Our approach follows the work presented in \cite{bunger2021low}. Before rewriting the all-at-once discretized system as a matrix equation, we make a matrix-splitting based on the special structure of the coefficient matrix. Following this splitting, we can obtain an equivalent matrix equation. We  reduce the system of the matrix equation into a  low-rank reduced matrix approximation equation under an appropriate approximation space. For the resulting small-size matrix equation, we  use the KPIK algorithm and then obtain the approximate solution.

The organization of the paper is as follows. In Section~\ref{sec:2}, we present the discretized linear system of the eddy current optimization problems \eqref{eddy-eq:1}-\eqref{eddy-eq:2} based on an all-at-once approach. In Section~\ref{sec:3}, we describe the MOR system and propose the new low-rank matrix equation method. Theoretical results  are investigated in detail in Section~\ref{sec:4}. Numerical experiments are carried out in Section~\ref{sec:5} to show the effectiveness of the new method. Finally, we draw some brief concluding remarks in Section~\ref{sec:6}.

The following notation is used throughout. Range($V$) denotes the space spanned by the columns of $V$. We use the notation $\|\cdot\|$  and $\|\cdot\|_F$ to indicate the 2-norm and the Frobenius norm  for vectors and the induced norm for matrices. The dimension of the identity matrix $I$ and the zero matrix $O$ will be omitted when it is easy to distinguish.

\section{Discretization and Optimization} \label{sec:2}

In this section, we make a detailed description of the eddy current optimization problem and then present the discretization of the all-at-once approach of the model.

We assume the control $u$ is weakly divergence-free.  The reluctivity $\nu$ is supposed to be uniformly positive and independent of $|\textbf{curl} y|$, i.e., we assume that the eddy current problem \eqref{eddy-eq:2} is linear.  Eddy current problems are essentially different for conducting ($\sigma>0$) and nonconducting regions ($\sigma=0$).

To gain uniqueness in the nonconducting regions, we need to regularize the state equation \eqref{eddy-eq:2} by introducing formal regularization operators ($i=1,2,3$) as: 

\[\mathcal{R}_i(\sigma):=
\left\{
\begin{array}{lr}
\sigma,&i=0,\\
\sigma,&i=1,\\
\max(\sigma,\varepsilon),&i=2,\\
\sigma,&i=3,
\end{array}
\right.\qquad\qquad
\mathcal{Q}_i(y):=
\left\{
\begin{array}{lrl}
0,&i=0&(\text{no\ regularization}),\\
Q(y),&i=1&(\text{exact}),\\
0,&i=2&(\text{conductivity}),\\
\varepsilon y,&i=3&(\text{elliptic}),
\end{array}
\right.
\]
where $i=0$ refers to without any regularization. Hence, the regularized problem can be stated as
\[
J(y,u)=\min_{(y,u)}\frac{1}{2}\int_0^T\int_{\Omega}|y-y_d|^2dxdt+\frac{\beta}{2}\int_0^T\int_{\Omega}|u|^2dxdt,
\]
subject to
\be\label{eddy-eq:3}
\left\{
\begin{array}{rlcl}
\mathcal{R}_i(\sigma)\frac{\partial y}{\partial t}+\textbf{curl}(\nu \textbf{curl}y)+\mathcal{Q}_i(y)&=u, &\text{in}&\Omega\times (0,T),\\
y\times n&=0,& \text{on}& \partial\Omega\times (0,T),\\
y&=y_0,& \text{on}& \Omega\times \{0\}.\\
\end{array}
\right.
\ee

Construct the Lagrangian functional
\[\mathcal{L}(y,u,p)=J(y,u)+\frac{1}{2} \int_{0}^T \int_{\Omega}(\mathcal{R}_i(\sigma)\frac{\partial y}{\partial t}+\textbf{curl}(\nu \textbf{curl}y)+\mathcal{Q}_i(y)-u)pdxdt\]
and obtain the necessary optimality conditions
\[
\left\{
\begin{array}{rl}
\nabla_p\mathcal{L}(y,u,p)&=0,\\
\nabla_y\mathcal{L}(y,u,p)&=0,\\
\nabla_u\mathcal{L}(y,u,p)&=0.\\
\end{array}
\right.
\]
It follows the system of PDEs: Find the state $y$, the costate $p$, and the control $u$ such that
\[
\left\{
\begin{array}{rlcl}
\mathcal{R}_i(\sigma)\frac{\partial y}{\partial t}+\textbf{curl}(\nu \textbf{curl}y)+\mathcal{Q}_i(y)-u&=0, &\text{in}&\Omega\times (0,T),\\
-\mathcal{R}_i(\sigma)\frac{\partial p}{\partial t}+\textbf{curl}(\nu \textbf{curl}p)+\mathcal{Q}_i(p)+y-y_d&= 0, &\text{in}&\Omega\times (0,T),\\
u&=\frac{1}{\beta}p,& \text{in}&\Omega\times (0,T),\\
y\times n&=0,& \text{on}& \partial\Omega\times (0,T),\\
p\times n&=0,& \text{on}& \partial\Omega\times (0,T),\\
y&=y_0,& \text{on}& \Omega\times \{0\},\\
p&=0,& \text{on}& \Omega\times \{0\}.\\
\end{array}
\right.
\]

Next, we discretize space and time using an all-at-once approach. We use the equal-order finite elements with spatial grid numbers being $N+1$ and to discretize the space. The temporal discretization is done with the time interval being split into $m_T$ intervals of length $\tau = \frac{T}{m_T}$. Using a rectangle rule, the discretization of \eqref{eddy-eq:1} and \eqref{eddy-eq:3} lead to
	\begin{equation}\label{eddy-eq:4}
	 \sum_{m=1}^{m_T}  \frac{\tau}{2} (y_{m} - y_{d,m})^T M (y_{m} - y_{d,m}) + \frac{\tau \beta}{2} u_m^T M u_m,
	\end{equation}
	where $y_m, y_{d,m}$ and $u_m$ are spatial discretizations of $y,y_d$ and $u$ of size $n$ for each time step $m = 1,\hdots,m_T$. Using an implicit Euler-scheme the discrete formulation of  \eqref{eddy-eq:2} will lead to
	\begin{align*}
	 \frac{M_{\sigma} (y_{m} - y_{m-1})}{\tau} + K y_m &= M u_m, \mbox{ for } m = 1,\hdots, m_T,\\
	 y_m\times n &= 0,
	\end{align*}
where the mass matrix $M\in \R^{n\times n}$, the conductivity matrix $M_{\sigma}\in \R^{n\times n}$ and stiffness matrix $K \in \R^{n\times n}$  arise from the finite element discretization of the following bilinear forms \cite{elman2005}:
\[M: \int_{\Omega}u\cdot vdx,\qquad M_{\sigma}: \int_{\Omega}\mathcal{R}_i(\sigma)u\cdot vdx, \]
\[K: \int_{\Omega}\nu\textbf{curl}u\cdot\textbf{curl}vdx+\int_{\Omega}\mathcal{Q}_i(u)vdx.\]
It then follows, for $m=1,2,\hdots, m_T$,
\[M_{\sigma}(y_m-y_{m-1})+\tau Ky_m-\tau Mu_m=0,\]
i.e.,
\[(M_{\sigma}+\tau K)y_m- M_{\sigma} y_{m-1}-\tau Mu_m=0.\]

By collecting  the discretizations of the variables in matrices $Y = [y_1,\hdots,y_{m_T}] \in \R^{n \times m_T}$ and  denote their vectorization by $\underline{Y}= \mbox{vec}(Y)$. The denotations $\underline{Y_d}$ and $\underline{U}$ are used respectively for $y_d$ and $u$. Therefore, we can rewrite the optimization problem in compact form as
	\[
	\min_{Y,U} \frac{\tau}{2} (\underline{Y}-\underline{Y_d})^T \mathcal{M} (\underline{Y}-\underline{Y_d}) + \frac{\tau \beta}{2} \underline{U}^T \mathcal{M} \underline{U},   \]
	such that
\[ \mathcal{N}_{\sigma}\underline{Y} - \tau \mathcal{M} \underline{U} = 0, \]
	where
	\begin{align*}
	 \mathcal{M} = \begin{bmatrix} M & & & \\ & M & & \\ & & \ddots & \\ & & & M \end{bmatrix},
\qquad
\mathcal{N}_{\sigma} = \begin{bmatrix} M_{\sigma} + \tau K & & & \\ -M_{\sigma} & M_{\sigma} + \tau K & & \\ & \ddots & \ddots \\ & & -M_{\sigma} & M_{\sigma} + \tau K \end{bmatrix},
	\end{align*}
with the mass matrix $M\in \R^{n\times n}$, the conductivity matrix $M_{\sigma}\in \R^{n\times n}$ and stiffness matrix $K \in \R^{n\times n}$ repeating $m_T$ times each.

Next, we need to address the system of equations derived from the first-order optimality conditions. In \cite{benzi2005numerical}, Benzi and  co-authors stated that an optimal solution must be a saddle point of the Lagrangian of the problem,
	 \[
	  \nabla \mathcal{L}(\underline{Y}^*, \underline{U}^*, \underline{P}^*) = 0.
	 \]
    The Lagrangian of this problem reads
	\[
	\mathcal{L}(\underline{Y},\underline{U},\underline{P}) = \frac{\tau}{2}(\underline{Y}-\underline{Y_d})^T \mathcal{M} (\underline{Y}-\underline{Y_d}) + \frac{\tau \beta}{2} \underline{U}^T \mathcal{M} \underline{U} + \underline{P}^T (\mathcal{N}_{\sigma} \underline{Y} - \tau \mathcal{M} \underline{U}).
	\]
	Thus, the optimal solution solves the following set of linear equations,
\begin{equation}\label{eddy-eq:5}
\left\{
	\begin{array}{rcl}
	\nabla_Y \mathcal{L}(\underline{Y},\underline{U},\underline{P)} &=& \tau \mathcal{M} (\underline{Y} - \underline{Y_d}) + \mathcal{N}_{\sigma}^T \underline{P} = 0, \\
	\nabla_U \mathcal{L}(\underline{Y},\underline{U},\underline{P)} &= &\tau \beta \mathcal{M} \underline{U} - \tau \mathcal{M}^T \underline{P} = 0,  \\
	\nabla_P \mathcal{L}(\underline{Y},\underline{U},\underline{P}) &= &\mathcal{N}_{\sigma} \underline{Y} - \tau \mathcal{M} \underline{U} = 0. 	\end{array}
\right.
\end{equation}
	
By introducing the auxiliary matrices
    \begin{align*}
	 \mathcal{I} = \begin{bmatrix} 1 & & \\ & \ddots & \\ & & 1 \end{bmatrix}\in \R^{m_T\times m_T} \mbox{ and } C = \begin{bmatrix} 1 & & & \\ -1 & 1 & & \\ & \ddots & \ddots & \\ & & -1 & 1 \end{bmatrix}\in \R^{m_T\times m_T},
	\end{align*}
     we can rewrite the above system matrices as Kronecker products $\mathcal{M} = \mathcal{I} \otimes M$ and $\mathcal{N}_{\sigma} = \mathcal{I} \otimes \tau K + C \otimes M_{\sigma}$. With this, our system of KKT conditions \eqref{eddy-eq:5}  becomes
     \begin{equation}\label{eddy-eq:6}
     \left\{
    \begin{array}{rcc}
    \tau (\mathcal{I} \otimes M) \underline{Y} + (\mathcal{I} \otimes \tau K^T + C^T \otimes M_{\sigma}^T) \underline{\Lambda} - \tau (\mathcal{I} \otimes M) \underline{Y_d} &=&0, \\
    \tau \beta (\mathcal{I} \otimes M) \underline{U} - \tau (\mathcal{I} \otimes M^T) \underline{\Lambda} &= &0, \\
    (\mathcal{I} \otimes \tau K + C \otimes M_{\sigma}) \underline{Y} - \tau (\mathcal{I} \otimes M) \underline{U} & =& 0.
    \end{array}
    \right.
    \end{equation}
	Because the mass matrix $M$ arising from a standard Galerkin method is always symmetric and positive definite, i.e., $M=M^T$ and $M^{-1}$ exists. Therefore, we can eliminate the second equation in Equation \eqref{eddy-eq:6} by setting
 $U = \frac{1}{\beta}\Lambda$ in the remaining two equations. It results in the following equations:
      \begin{equation}\label{eddy-eq:7}
    \begin{pmatrix}
    \tau \mathcal{M}&\sqrt{\beta}\mathcal{N}_{\sigma} ^T\\
    \sqrt{\beta}\mathcal{N}_{\sigma} &-\tau\mathcal{M}
    \end{pmatrix}
    \begin{pmatrix}
     \underline{Y}\\ \frac{1}{\sqrt{\beta}}\underline{\Lambda}
    \end{pmatrix}=
    \begin{pmatrix}
     \tau (\mathcal{I} \otimes M) \underline{Y_d}\\0
    \end{pmatrix}.
     \end{equation}

\section{The splitting-based KPIK method}\label{sec:3}

In this section, we present the splitting-based KPIK method for the linear system \eqref{eddy-eq:7}. By substituting $\mathcal{M}= I\otimes M$,  $\mathcal{N}_{\sigma} = \mathcal{I} \otimes \tau K + C \otimes M_{\sigma}$ and $M_{\sigma}=\sigma M$, and $K=K^T$, we can rewrite the linear system \eqref{eddy-eq:7} as
      \begin{equation*}
    \begin{pmatrix}
    \tau \mathcal{I}\otimes M&\sqrt{\beta}(\mathcal{I} \otimes \tau K + C \otimes \sigma M)^T\\
    \sqrt{\beta}(\mathcal{I} \otimes \tau K + C \otimes \sigma M) &-\tau \mathcal{I}\otimes M
    \end{pmatrix}
    \begin{pmatrix}
     \underline{Y}\\ \frac{1}{\sqrt{\beta}}\underline{\Lambda}
    \end{pmatrix}=
    \begin{pmatrix}
     \tau (\mathcal{I} \otimes M) \underline{Y_d}\\0
    \end{pmatrix},
     \end{equation*}
i.e.,
      \begin{equation*}
      \Big(
    \begin{pmatrix}
    \tau \mathcal{I}\otimes M&\sqrt{\beta} C^T \otimes \sigma M\\
    \sqrt{\beta} C \otimes \sigma M &-\tau \mathcal{I}\otimes M
    \end{pmatrix}+
        \begin{pmatrix}
   0&\sqrt{\beta}\mathcal{I} \otimes \tau K \\
    \sqrt{\beta}\mathcal{I} \otimes \tau K  &0
    \end{pmatrix}\Big)
    \begin{pmatrix}
     \underline{Y}\\ \frac{1}{\sqrt{\beta}}\underline{\Lambda}
    \end{pmatrix}
    =
    \begin{pmatrix}
     \tau (\mathcal{I} \otimes M) \underline{Y_d}\\0
    \end{pmatrix}.
     \end{equation*}
We can further rewrite the above matrix equation as
 \begin{equation}\label{eddy-eq:8}
      \Big(
    \begin{pmatrix}
    \tau \mathcal{I}& \sigma\sqrt{\beta} C^T  \\
    \sigma\sqrt{\beta} C   &-\tau \mathcal{I}
    \end{pmatrix}\otimes M +
        \begin{pmatrix}
   0&\tau \sqrt{\beta}\mathcal{I} \\
    \tau\sqrt{\beta}\mathcal{I}   &0
    \end{pmatrix}\otimes K\Big)
    \begin{pmatrix}
     \underline{Y}\\ \frac{1}{\sqrt{\beta}}\underline{\Lambda}
    \end{pmatrix}
    =
    \begin{pmatrix}
     \tau (\mathcal{I} \otimes M) \underline{Y_d}\\0
    \end{pmatrix}.
     \end{equation}
      Using the relation
	\[
     (W^T \otimes V) \mbox{vec}(X) = \mbox{vec}(VXW),
    \]
    we rewrite the linear system \eqref{eddy-eq:8} in a matrix-equation form as
   \begin{equation}\label{eddy-eq:matrixeq}
      M X
    \begin{pmatrix}
    \tau \mathcal{I}& \sigma\sqrt{\beta} C^T  \\
    \sigma\sqrt{\beta} C   &-\tau \mathcal{I}
    \end{pmatrix}  + K X
        \begin{pmatrix}
   0&\tau \sqrt{\beta}\mathcal{I} \\
    \tau\sqrt{\beta}\mathcal{I}   &0
    \end{pmatrix}
    =[\tau MY_d\quad 0]\in \R^{n\times 2m_T}.
     \end{equation}
 Multiplying $M^{-1}$ by the left and
    \[\begin{pmatrix}
   0&\tau \sqrt{\beta}\mathcal{I} \\
    \tau\sqrt{\beta}\mathcal{I}   &0
    \end{pmatrix}^{-1}\]
    by the right of \eqref{eddy-eq:matrixeq} on both sides , respectively, we can further obtain
        \begin{equation*}
     M^{-1}K X+ X
    \begin{pmatrix}
    \frac{\sigma}{\tau}C^T& \frac{1}{\sqrt{\beta}}\mathcal{I} \\
     -\frac{1}{\sqrt{\beta}}\mathcal{I} & \frac{\sigma}{\tau} C
    \end{pmatrix}
    =M^{-1}[\tau MY_d\quad 0]\begin{pmatrix}
   0&\frac{1}{\tau \sqrt{\beta}}\mathcal{I} \\
    \frac{1}{\tau \sqrt{\beta}}\mathcal{I}   &0
    \end{pmatrix},
     \end{equation*}
     i.e.,
             \begin{equation}\label{eddy-eq:9}
     M^{-1}K X+ X
    \begin{pmatrix}
    \frac{\sigma}{\tau}C^T& \frac{1}{\sqrt{\beta}}\mathcal{I} \\
     -\frac{1}{\sqrt{\beta}}\mathcal{I} & \frac{\sigma}{\tau} C
    \end{pmatrix}
    =[0\quad \frac{1}{\sqrt{\beta}}Y_d ]:=R \in \R^{n\times 2m_T},
     \end{equation}
where $X=[Y\quad\frac{1}{\sqrt{\beta}}\Lambda]\in \R^{n\times 2m_T}$.

 Next, we focus on the low-rank approximation solution of the Sylvester matrix equation
\begin{equation*}
AX+XB=R,
\end{equation*}
    where the left-hand coefficient matrix $A=  M^{-1}K$ has size $n\times n$ and the coefficient matrix
    \begin{equation}\label{eddy-eq:10}
    B=\begin{pmatrix}
    \frac{\sigma}{\tau}C^T& \frac{1}{\sqrt{\beta}}\mathcal{I} \\
     -\frac{1}{\sqrt{\beta}}\mathcal{I} & \frac{\sigma}{\tau} C
    \end{pmatrix}
    \end{equation}
    has size $2m_T\times 2m_T$,  while the right-hand one has size $n \times 2m_T$, so that $X \in {\mathbb R}^{n\times 2m_T}$.

Suppose that there exists a low-rank approximation of the desired state  as
    \begin{equation*}
     Y_d \approx Y_1 Y_2^T,
    \end{equation*}
    with $Y_1 \in \R^{n \times r}$, $Y_2 \in \R^{m_T \times r}$ and $r < m_T$ of low column and row rank, respectively.   Then the approximate low-rank decomposition of the right-hand side $R\approx R_1R_2^T$ is given by
    \begin{equation}\label{eddy-eq:decomp}
    R_1=\frac{1}{\sqrt{\beta}} Y_1\in {\mathbb R}^{n\times r}\quad \text{and}\quad
    R_2=
    \begin{pmatrix}
    O_{m_T\times r}\\ Y_2
    \end{pmatrix}\in {\mathbb R}^{2m_T\times r},
    \end{equation}
 with $R_1$ and $R_2$ being low column and row rank, respectively. Here, $O_{m_T\times r}$ denotes a  zero matrix of size $m_T\times r$. We will omit the subscript when it is easy to distinguish in the following of this paper.

Because the solution matrix $X \in \R^{n\times 2m_T}$ would be dense and potentially very large, then we need to find an appropriate approximation space to exploit the new setting of the solution matrix.  Suppose that there exists a low-rank reduced matrix approximation $Z\in \R^{p \times 2m_T}$ such that $X \approx V_p Z$, where the orthonormal columns of $V_p \in \R^{n \times p}$ generate the approximation space. Hence, we can construct a reduced version of the matrix equation \eqref{eddy-eq:10}.


Denote the reduced $p\times p$ coefficient matrices as
$A_{r} := V_p^T A V_p$ and set $R_{1,r} = V_p^T R_1 \in \R^{p \times r}$. Then the resulting reduced equation can be written as
    \begin{equation} \label{eddy-eq:11}
    A_{r} Z +   Z B  = R_{1,r} R_2^T,
    \end{equation}
    which has the same structure as the original matrix equation but with its size reduced to $p \times 2m_T$. Using the relation in \eqref{eddy-eq:10}, we get the small linear system of equations
    \begin{equation} \label{eddy-eq:12}
     \big ( (\mathcal{I}_{2m_T} \otimes A_{r}) + (B^T \otimes \mathcal{I}_p)  ) \underline{Z} = \underline{R_{1,r}}\underline{R_2}^T,
    \end{equation}
    with $\underline{Z} = \mbox{vec}(Z)$, $\underline{R_{1,r}}=\mbox{vec}(R_{1,r})$, and  $\underline{R_2}=\mbox{vec}(R_2)$. 
    
    Because of the small subspace size $p \ll n$, we can either use a direct method or an iterative method to solve this system of equations, which is significantly easier to solve than the original system of equations.  If the obtained approximate solution $V_pZ$ is not sufficiently good, then the space can be expanded.  Hence, a new approximation can be constructed, giving rise to an iterative method. More details can be found in \cite{stoll2015low,breiten2016fast}.

 At the end of this section, we give a detailed implementation  of the SKPIK  method in Algorithm \ref{eddy-algo:2}.

  \begin{algorithm}
\caption{The splitting-based KPIK  method}
\label{eddy-algo:2}
\begin{algorithmic}[1]
\STATE Given a tolerance $\epsilon$, a maximum number of iteration $\text{IT}_{max}$, a spatial grid points number $n$, a temporal grid points number $m_T$ and a tolerance for the truncated singular value
decomposition (SVD) $\varepsilon$. The step-size parameter $\tau=\frac{1}{m_T}$.
\STATE Inputs: $M\in \R^{n\times n}$ a mass matrix, $K\in \R^{n\times n}$ a stiffness matrix, $Y_d\in \R^{n\times m_T}$ a given desired state vector,  $\sigma>0$ a conducting constant and $\beta>0$ a cost parameter.
\STATE  Set $A\approx M^{-1}K$, $B$ given by \eqref{eddy-eq:10}, $R\approx R_1R_2^T$ within truncated SVD $\varepsilon$, i.e., dropping the singular values of $R$ smaller than $\varepsilon$ \cite{stoll2015low}. $U_1=gram\_sh(R, A^{-1}R)$, $W_1=gram\_sh(R, B^{-1}R)$,  $U_0=\varnothing$ and $W_0=\varnothing$.
\FOR{$m=1,2,\ldots,\text{IT}_{max}$}
\STATE $U_m=[U_{m-1},U_m]$, $W_m=[W_{m-1},W_m]$.
\STATE Set $T^A_m=W_{m}^TAU_m$, $T^B_m=W_{m}^TBU_m$ and $R_1^{(m)}=U_mR_1$, $R_2^{(m)}=W_m^TR_2$.
\STATE Solve $T_m^AY+Y(T_m^B)^T+R_1R_2^T=0$ and set $Y_m=Y$.
\STATE If converged, the $X_m=U_mY_mW_m^T$ and stop.
\STATE Set $U_m^{(1)}$: first $s$ columns of $U_m$, $U_m^{(2)}$: second $s$ columns of $U_m$; $W_m^{(1)}$: first $s$ columns of $W_m$, $W_m^{(2)}$: second $s$ columns of $W_m$.
\STATE $U_{m+1}'=[AU_m^{(1)},A^{-1}U_m^{(2)}]$ and $W_{m+1}'=[BW_m^{(1)},B^{-1}W_m^{(2)}]$.
\STATE $\hat{U}_{m+1}\leftarrow$ orthogonalize $U_{m+1}'$ w.r.t. $U_m$ and $\hat{W}_{m+1}\leftarrow$ orthogonalize $W_{m+1}'$ w.r.t. $W_m$.
\STATE $U_{m+1}=gram\_sh(\hat{U}_{m+1})$ and $W_{m+1}=gram\_sh(\hat{W}_{m+1})$.
\IF{$\frac{\|A U_mY_m W_m^T + U_mY_m W_m^T B^T + R_1 R_2^T\|}{\|R_1 R_2^T\|}< \epsilon$}
\STATE $X_1=U_mY_m$ and $X_2=W_m$. Break.
\ENDIF
\ENDFOR
\STATE The low-rank approximation solution $X$ is given by $X\approx X_1X_2^T$.\\
 \COMMENT {$\rhd$ The function ``gram\_sh" performs the modified Gram-Schmidt orthogonalization, more details can be found in  \cite{golub2013matrix}. }
\end{algorithmic}
\end{algorithm}

\begin{rem}\label{eddy-rem:1}
Compared with the low-rank MINRES method proposed in \cite{stoll2015low}, the SKPIK  method in this work needs one to get the solution of size $n\times 2m_T$ from the Sylvester equation \eqref{eddy-eq:9} once, while the low-rank MINRES method, preconditioned by  the block diagonal preconditioner with the approximate Schur complement proposed in \cite{stoll2015low}, needs one to obtain the solutions of size $n\times m_T$ from two different Sylvester equations. However, the low-rank MINRES method also needs one to solve  sublinear systems with the coefficient matrices being $\mathcal{M}$ twice. Hence, the SKPIK method has a smaller workload than the low-rank MINRES method.
\end{rem}

\begin{rem}\label{eddy-rem:2}
As the stiffness matrix $K$ is SPSD and the mass matrix $M$ is SPD, then following the idea in \cite{shank2013krylov}, we can also shift the matrix equation $AX+XB=R_1R_2^T$ of \eqref{eddy-eq:9} as a equivalent matrix equation
$(A+s I)X+X(B-s I)=R_1R_2^T$, where $s>0$ is a given shifted parameter, such that $A+s I=M^{-1}(K+sM)$ is better conditioned than the matrix $M^{-1}K$.
\end{rem}

\section{The existence of the low-rank solution}\label{sec:4}
In this section, we  give the existence of the low-rank solution. The existence of the low-rank approximant to a Sylverster equation is given in  the following lemma.

\begin{lem}\label{eddy-lem:1}(Existence of a low rank approximant, Corollary 2 in \cite{grasedyck2004existence})
Let $A\in \mathbb{C}^{n\times n}$ and $B\in \mathbb{C}^{n\times n}$ be matrices with spectrum $\sigma(A)$ and $\sigma(B)$ as in the rectangular case or the triangular case with constant $\mu$ ($\mu\geq1$ in the triangular case),  $\Gamma_A$ and  $\Gamma_B$ are paths of index 1 around the spectrum of $A$ and $B$ with $\Lambda$, $\lambda$, paths $\Gamma_A$, $\Gamma_B$ and a partitioning $\Gamma_B=\dot{\bigcup}_{j=0}^{k_{\sigma}-1}\Gamma_{B,j}$. For all $j=0,1,2,\ldots,k_{\sigma}-1$, let $\eta_j$ be the center of a part $\Gamma_{B,j}$, $\xi\in \Gamma_A$ and $\eta\in \Gamma_{B,j}$. We define
\[g_{i,j}(\xi):=(\xi-\eta_j)^{-1-i}\quad\text{and}\quad h_{i,j}(\eta):=(\eta-\eta_j)^i,\quad 0\leq i<k,\quad 0\leq j<k_{\sigma},\]
and
\[\kappa_A:=\frac{1}{2\pi}\oint_{\Gamma_A}\|(\xi I-A)^{-1}\|_Fd\xi,\quad \kappa_B:=\frac{1}{2\pi}\oint_{\Gamma_B}\|(\eta I-B)^{-1}\|_Fd\eta,\]
\[k_{\varepsilon}:=\Big\lceil\log_2\big(1+\frac{6(\|A\|_F+\|B\|_F)\kappa_A\kappa_B}{\lambda\varepsilon}\big)\Big\rceil, \]
where $\varepsilon\in(0,1)$ is given. Then for each right-hand side matrix $R\in \mathbb{C}^{n\times m}$ of rank $r_R$, the matrix
\[\tilde{X}:=\frac{1}{4\pi^2}\sum_{j=0}^{k_{\sigma}-1}\sum_{j=0}^{k_{\varepsilon}-1}\big(\oint_{\Gamma_A}(\xi I-A)^{-1}g_{i,j}(\xi)d\xi\big)R\big(\oint_{\Gamma_B}(\eta I-B)^{-1}h_{i,j}(\eta)d\eta\big)\]
approximates the solution $X$  to \eqref{eddy-eq:9} by
\[\|X-\tilde{X}\|_F\leq \varepsilon\|X\|_F.\]
Besides, the rank of $\tilde{X}$ is bounded by $r_Rk_{\sigma}k_{\varepsilon}$ with
\begin{equation*}
k_{\sigma}=
\left\{
\begin{array}{lll}
&\mathcal{O}(\frac{\mu}{\lambda}+\log_2(2+\frac{\Lambda}{\lambda}))& (rectangular\ case),\\
&\mathcal{O}(\mu \log_2(2+\frac{\Lambda}{\lambda}))& (triangular case).
\end{array}
\right.
\end{equation*}
\end{lem}

Lemma \ref{eddy-lem:1} gives the existence of the low-rank solution to the equation \eqref{eddy-eq:9}. However, due to the complexity of the problems, it is difficult to estimate the exact rank of the discretization system from the eddy current optimal problem. The bound of the rank can be obtained from Lemma \ref{eddy-lem:1} and it is present in the following theorem.

\begin{thm}\label{eddy-thm:1}
Given the eddy current optimization problems \eqref{eddy-eq:1}-\eqref{eddy-eq:2}. Let $h$ and $\tau$ be the spatial step and temporal step, respectively. By making use of the all-at-once based approach, we have the discretized system \eqref{eddy-eq:7}, which is equivalent to the matrix equation \eqref{eddy-eq:9}. Suppose that the right-hand side matrix $R$  is given in the low-rank form \eqref{eddy-eq:decomp} with $r_R$ being the rank, and the solution is approximated in the form $\tilde{X} = V_p Z (\approx X)$ up to an accuracy $\varepsilon$. Let $M$ and $K$ be the mass matrix and the stiffness matrix described previously.  Then the rank $r_X$ of the solution $\tilde{X}$ is bounded by
\[\bar{r}_X=\mathcal{O}((\log \frac{1}{\varepsilon} +\log( \frac{1}{h^2}+\frac{1}{\tau}))^2r_R).\]
\end{thm}
\begin{proof}
We can rewrite the resulting matrix equation \eqref{eddy-eq:9} into a equivalent linear equation form as
\[(I_n\otimes M^{-1}K+B\otimes I_{2m_T})\underline{X}= \underline{R},\]
where the underline of a matrix denotes its vectorization. Denote by $\mathcal{A}=I_n\otimes M^{-1}K+B\otimes I_{2m_T}$, then by using the results in \cite{grasedyck2004existence,dolgov2017low}, we can approximate the inverse of $\mathcal{A}$ in the low-rank form by the exponential quadrature, i.e., for given $N$ and $k$, denote by $t_k=exp(\frac{k\pi}{\sqrt{N}})$, $c_x=\frac{t_k\pi}{\sqrt{N}}$, then
\[\mathcal{A}^{-1}=(I_n\otimes M^{-1}K+B\otimes I_{2m_T})^{-1}\approx \Sigma_{k=-N}^N c_k\exp(-t_kB)\otimes \exp(-t_k M^{-1}K),\]
where the accuracy is estimated by $\mathcal{O}(\|\mathcal{A}\|_2\cdot e^{-\pi\sqrt{2N}})$, provided that $\|\mathcal{A}^{-1}\|=\mathcal{O}(1)$.

Therefore, the rank of $\mathcal{A}^{-1}$ is estimated by $\mathcal{O}((\log\frac{1}{\varepsilon}+\log \text{cond} \mathcal{A})^2)$. Moreover, as $\text{cond}(\mathcal{A})=\mathcal{O}(h^{-2}+\tau^{-1})$, then the rank of $\tilde{X}$ is bounded by
\[\mathcal{O}((\log\frac{1}{\varepsilon}+\log(\frac{1}{h^2}+\frac{1}{\tau}))^2 r_R).\]
\end{proof}

\section{Numerical Experiments}\label{sec:5}

We now present the performance and flexibility of the splitting-based KPIK method (denoted by `` SKPIK ") on two  examples for the eddy current constrained optimization problems. All experiments were run on a desktop computer with an Intel(R) Core(TM) i7-10710U CPU @ 1.10GHz   1.61 GHz with 16 GB of RAM. The lowest order linear N\'{e}d\'{e}lec edge element, which is defined on 2D triangles \cite{nedelec1980mixed}  and 3D tetrahedra \cite{nedelec1986new},  are used to discretize the state variable, the control variable, and the adjoint variable in our experiments. The MATLAB package of \cite{anjam2015fast} is used to construct the relevant matrices. To show numerically the feasibility and effectiveness of the new method, we report the results of all the proposed methods in the sense
 of iteration step (denoted as ``IT"), elapsed CPU time in seconds (denoted as ``CPU"),  and relative residual error (denoted as ``RES").  Regarding the low-rank  approximation methods, we  report the rank (denoted as ``r") of the final solutions additionally.

We test all the methods for two examples, i.e., a 2D problem and a 3D problem. In both examples, the mass matrix $M$ is symmetric positive definite. Therefore, the coefficient matrix $\mathcal{A}$ of matrix-vector form of the linear equations \eqref{eddy-eq:5}:
\begin{equation}\label{eddy-eq:5-matrixform}
\mathcal{A}x:=
\begin{pmatrix}
\tau \mathcal{M}& 0& \mathcal{N}_{\sigma}^T\\
0&\tau\beta \mathcal{M} &-\tau \mathcal{M}\\
\mathcal{N}_{\sigma} & -\tau \mathcal{M} &0
\end{pmatrix}
\begin{pmatrix}
\underline{Y}\\
\underline{U}\\
\underline{P}
\end{pmatrix}=
\begin{pmatrix}
\tau \mathcal{M} \underline{Y}_d\\
0\\
0
\end{pmatrix}=:b
\end{equation}
 is symmetric, then we can consider  the low-rank approximate combining with the  MINRES method \cite{stoll2015low}. To implement the low-rank MINRES method efficiently, one always apply an appropriate preconditioner, e.g., the  block-diagonal  preconditioner with the Schur complement as
       \begin{equation*}
   \mathcal{P}=
    \begin{pmatrix}
    \tau \mathcal{M}&&\\
    &\tau\beta \mathcal{M}&\\
    &&\mathcal{S}
    \end{pmatrix},
    \end{equation*}
    where the Schur complement $\mathcal{S}=\frac{1}{\tau}\mathcal{N}_{\sigma}^T\mathcal{M}^{-1}\mathcal{N}_{\sigma}+\frac{\tau}{\beta}\mathcal{M}$. However, the main difficulty to use the above preconditioner is the heavy workload of solving the sub-linear system with the coefficient matrix being  the Schur complement $\mathcal{S}$. Hence, we will compare the preconditioner using the approximation of  the Schur-complement $\hat{\mathcal{S}}\approx \mathcal{S}$, which is presented by Pearson and co-authors in  \cite{pearson2012new,pearson2012regularization}, as
        \begin{equation}\label{eddy-eq:13}
   \widehat{\mathcal{P}}=
    \begin{pmatrix}
    \tau \mathcal{M}&&\\
    &\tau \beta\mathcal{M}&\\
    &&\hat{\mathcal{S}}
    \end{pmatrix},
    \end{equation}
with $\hat{\mathcal{S}}=\frac{1}{\tau}(\mathcal{N}_{\sigma}+\frac{\tau}{\sqrt{\beta}}\mathcal{M})\mathcal{M}^{-1}(\mathcal{N}_{\sigma}+\frac{\tau}{\sqrt{\beta}}\mathcal{M})^T$.
Therefore, we compare our new method with the low-rank MINRES method with the preconditioner $\widehat{\mathcal{P}}$ in \eqref{eddy-eq:13} for solving the discretizetion system \eqref{eddy-eq:5-matrixform} by all-at-once approach, namely `` LRMINRES ". Details about the LRMINRES  algorithm can be found in \cite{stoll2015low}.

We want to point out that when we need to solve the sublinear systems with the coefficient matrix being sparse and symmetric positive definite, e.g.,  $M$, $K$, and so on, we use the sparse Cholesky factorization incorporated with the symmetric approximation minimum degree reordering. To do so, we use $\text{symamd.m}$ command in the MATLAB toolbox. The same way to use the above command can be found in \cite{bai2003hermitian} and it has been proven to be very effective in solving large sparse symmetric positive definite linear equations.

During the comparison, we also used the non-all-at-once approach to discretize the eddy current optimization problem.
Then the full-rank  MINRES algorithm \cite{paige1975solution} is used to solve the resulting discretization systems step-by-step. Or equivalently, at each time step, we use the MINRES algorithm with the block-diagonal  preconditioner $P$, which is defined by
\begin{equation*}
P=
\begin{pmatrix}
    \tau M&&\\
    &\tau \beta M&\\
    &&\hat{S}
\end{pmatrix},
\end{equation*}
where   $\hat{S}=\frac{1}{\tau}(K+\frac{\tau}{\sqrt{\beta}} M)M^{-1}(K+\frac{\tau}{\sqrt{\beta}} M)$ is approximation to the Schur complement $S=\frac{1}{\tau}KM^{-1}K+\frac{\tau}{\beta}M$. The corresponding algorithm is denoted as `` FMINRES  " in our experiments.

  The tolerance for all methods is set to be $10^{-6}$. That is to say, all iteration processes are terminated when the current relative residuals satisfy
\[
\text{RES}:=\frac{\|b-\mathcal{A}x^{(k)}\|_2}{\|b\|_2}\leq 10^{-6},
\]
or the methods do not reach convergence when the  numbers of iteration steps reach the maximum number $\text{IT}_{\max}=500$, or the computing time of the corresponding method exceeds 1000 in seconds. $x^{(0)}=0$ is the initial guess and $x^{(k)}$  is the $k$th iterates of the corresponding iteration processes, respectively. If the current iterates can not reach the above  tolerance within the given maximum iteration number or the above limited time, we denote the results as `-' in the following tables.

To find the low-rank representation $X=V_pZ$ of the solution to the matrix equation \eqref{eddy-eq:9}, we perform the skinny QR factorization of both matrices, i.e., $V_p$ and $Z$. The rank of the low-rank approximation solution is produced by dropping small singular values (depending on some tolerance, denoted as ``truncation tolerance'' in our experiments) and then it leads to a low-rank approximation. The MATLAB function $\text{svds}$ is used directly to compute the singular value decomposition (SVD) of $V_pZ$.  According to \cite{stoll2015low}, we use the truncation tolerance $10^{-10}$ throughout our numerical experiments. More details about the way to compute the truncated SVD can be found in \cite{stoll2015low}.

It is worth mentioning that the iteration counts of the SKPIK method we report in all the tables refer to the number of iteration counts of the extended Krylov subspace method, i.e., the KPIK algorithm \cite{simoncini2007new}. While the number of iteration counts of the LRMINRES method refers to the number of iteration counts of the MINRES algorithm.  The number of the iteration counts of the FMINRES method refers to the average of the numbers of iteration counts, i.e., the sum of the required iterations on all-time steps divided by the number of time steps $m_T$. Besides, when we solve the systems regarding the coefficient matrices being the Schur-complement approximation  in the LRMINRES method, we employ the inexact KPIK with a fixed number of steps, i.e.,  6 steps according to \cite{stoll2015low}.

\begin{example}\label{eddy-ex:1}
{	\rm We consider $\Omega=[0,1]^2$ and  $\nu=1$. We split the domain into two parts across the diagonal with $\Omega_1=\{x\in \Omega | x_1>x_2\}$ and $\Omega_2 = \Omega\setminus \Omega_1$. The desired state is given by
\[
y_d(x,t)|_{\Omega_1}=
\begin{pmatrix}
\sin(2\pi x_1)+2\pi\cos(2\pi x_1)(x_1-x_2)\\
\sin((x_1-x_2)^2(x_1-1)^2x_2-\sin(2\pi x_1))
\end{pmatrix} \quad \text{and} \quad y_d(x,t)|_{\Omega_2}=0.
\]
}
\end{example}

 In this example, our experiments are performed for the final time $T=1$ by varying the numbers of time steps as $m_T=100, 200, 400,  800, 1600$  and $3200$. We first test all the methods regarding  the numbers of spatial discretization nodes being $n=3136$, as well as a large range of problem parameters concerning $\beta=10^{-2}$, $10^{-4}$, $10^{-6}$ and $10^{-8}$, $\sigma=10^{-5}$, $10^{-3}$, $10^{-1}$, $10^1$, $10^3$ and $10^5$.

 Firstly, we test a small number of time steps and a small size of space discretization, i.e., $m_T=100$ and $n=3136$. The results are listed in Figure \ref{eddy-fig:1}, concerning different choices of the parameters $\beta$ and $\sigma$, where on the left is the SKPIK method, in the middle is the LRMINRES method, and on the right is the FMINRES method.

   \begin{figure}[htb]
	\includegraphics[width=0.3\textwidth]{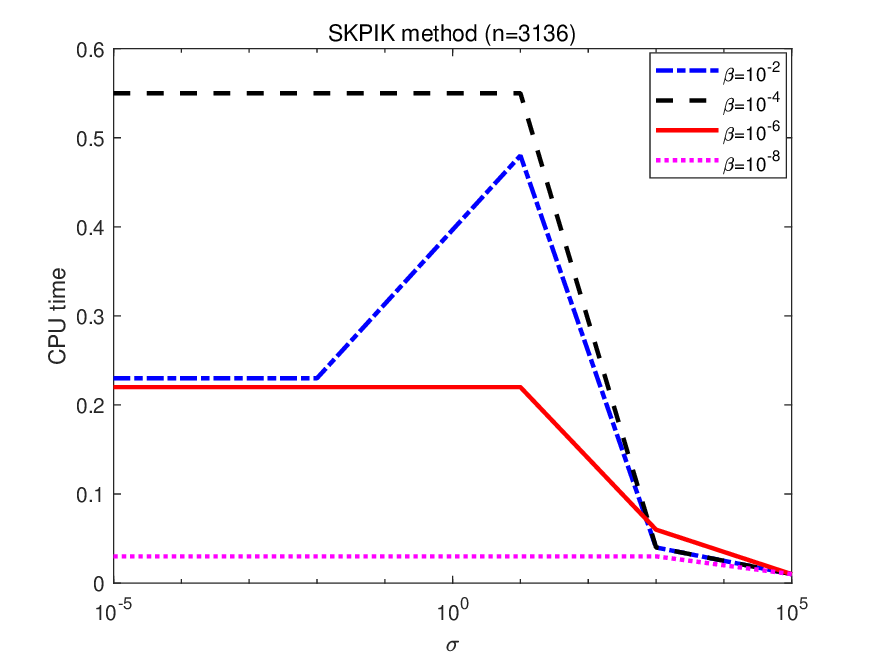}
\includegraphics[width=0.3\textwidth]{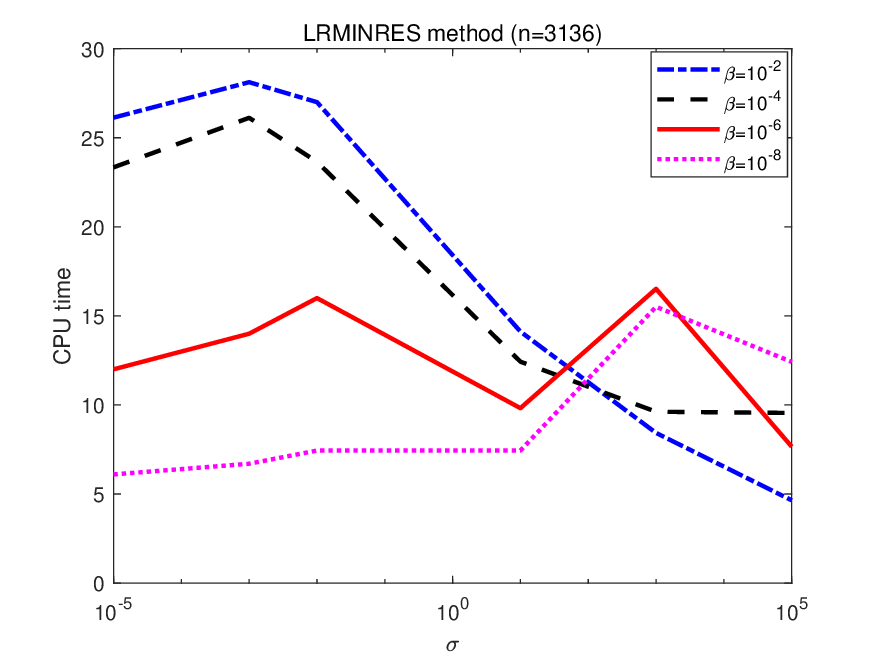}
\includegraphics[width=0.3\textwidth]{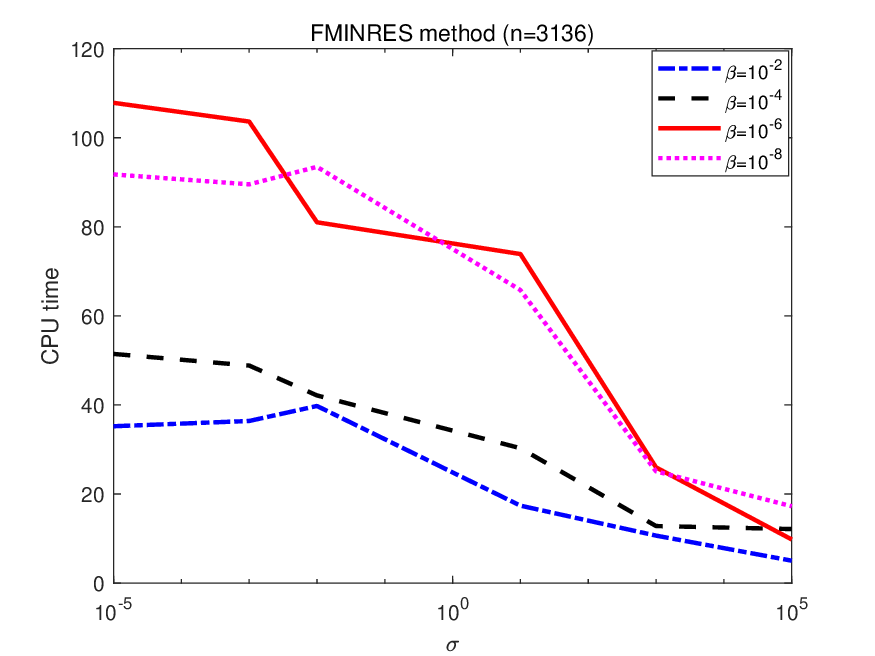}
	\centering\caption{The computing time against different values of $\sigma$ and $\beta$ for all the proposed methods ($m_T=100$).} \label{eddy-fig:1}
\end{figure}

 We further test all the methods by varying the numbers of time steps as $m_T=$100, 200, 400, 800, 1600, 3200 for the  size of space discretization  $n=3136$. Then we plot the computing time against the number of time steps in Fig. \ref{eddy-fig:2} by fixing $\sigma =10$, where again on the left is the SKPIK method, in the middle is the LRMINRES method, and on the right is the FMINRES method.

 \begin{figure}[htb]
	\centering
	\includegraphics[width=0.3\textwidth]{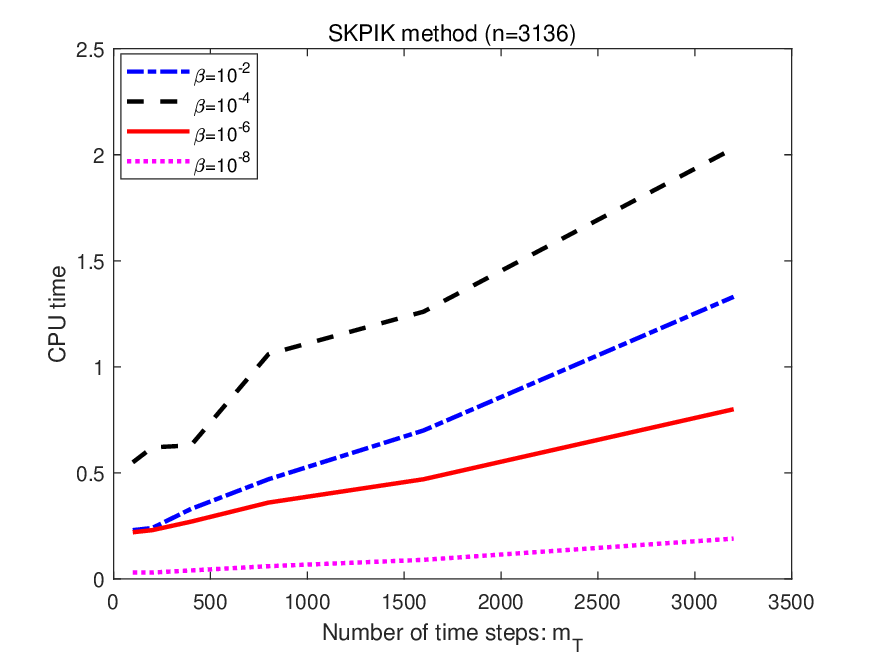}
\includegraphics[width=0.3\textwidth]{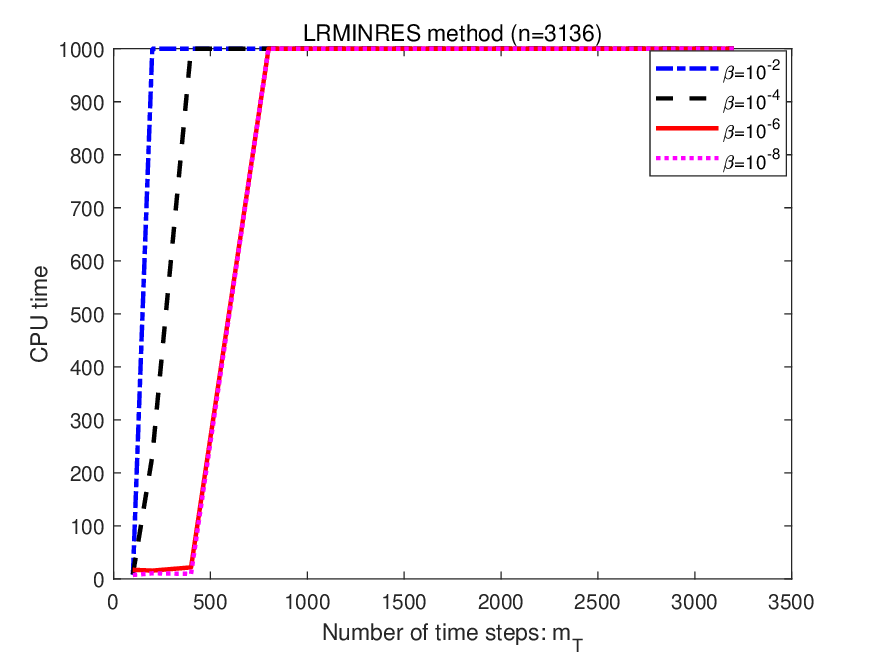}
\includegraphics[width=0.3\textwidth]{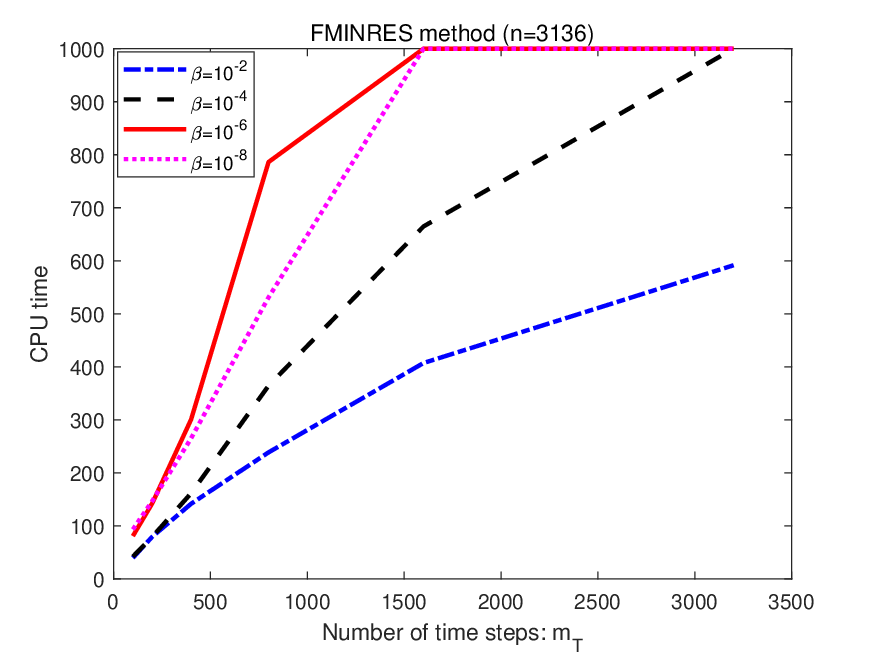}
	\centering\caption{The computing time against different time steps $m_T$ for all the proposed methods.} \label{eddy-fig:2}
\end{figure}

 From Figures \ref{eddy-fig:1}-\ref{eddy-fig:2}, we see the  FMINRES  method can solve the optimization problem successfully for small dimensional problems because of the efficiency of the diagonal-block preconditioning technique and the MINRES algorithm.  However, we also see that the FMINRES  method has always consumed lots of time to obtain the solution, because the  FMINRES  method needs to obtain the solution at each time step before solving the system of the next time step, and also needs to prepare for the coefficient matrices and the right-hand side vectors for the next time step. Besides, from Fig. \ref{eddy-fig:2},  we also find that the computing time needed by the  FMINRES  method increases rapidly as the number of time steps increases.

 Meanwhile, the  LRMINRES  method seems to be more efficient than the  FMINRES  method for the small number of time steps and the small size of space discretization in Figure \ref{eddy-fig:1}. Particularly, when $m_T=100$ and 200, and the parameters $\beta$ and $\sigma$ are small enough, the  LRMINRES  method performs very well. But when $m_T=400$, 800, 1600, and 3200, the iteration numbers to get the solution needed by the LRMINRES method vary greatly corresponding to different problem parameters $\beta$ and $\sigma$. These results are consistent with the results obtained by Stoll and co-authors \cite{stoll2015low}.

 From all the results plotted in these figures, we see the SKPIK method always performs the best. It needs the least computing time. Besides, as the number of time steps increases, the computing time of the SKPIK method increases very gently. To illustrate this, we further plot in Fig. \ref{eddy-fig:3} the computing time of the SKPIK method against the number of time steps for the number of spatial discretization nodes being $n=12416$.

  \begin{figure}[htb]
	\centering
	\includegraphics[width=0.45\textwidth]{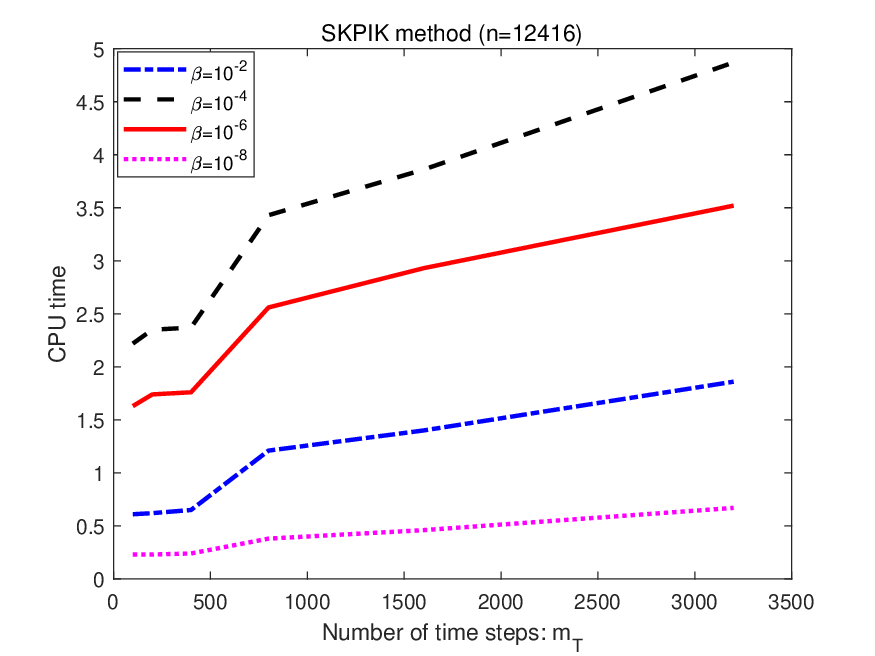}
\includegraphics[width=0.45\textwidth]{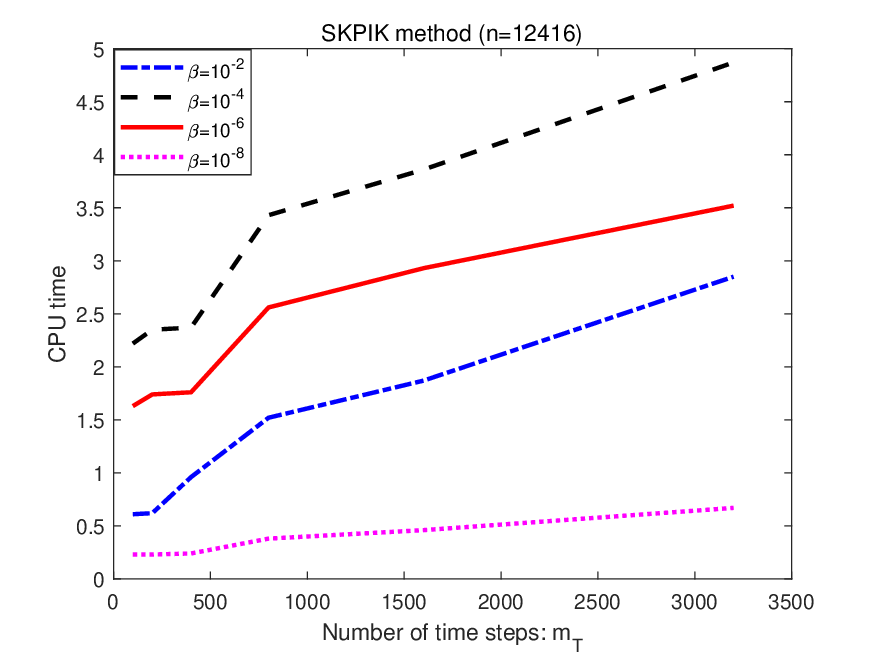}
	\centering\caption{The computing time of the SKPIK method against different time steps $m_T$ ($\sigma=10^{-1}$(left), $\sigma=10^1$(right)).} \label{eddy-fig:3}
\end{figure}

Therefore,  as seen in Figures \ref{eddy-fig:1}-\ref{eddy-fig:3}, increasing the discretization size barely impacts the resulting subspace sizes.  The computing time needed by the  FMINRES   method increases rapidly as the number of time steps increases. Compared with the  FMINRES  method, the  LRMINRES  method seems to be less impacted concerning the number of time steps. However, the  LRMINRES  method depends greatly on the choices of the parameters $\beta$, $\sigma$ as well as the number of time steps.

 From all the above results, we can see that the SKPIK  method  shows great robustness concerning the discretization sizes $n$ and $m_T$ as well as the control parameter $\beta$ and the conductivity $\sigma$. Additionally, the time needed by the SKPIK method to solve the optimization problems increases considerably slowly.

To show more details about the computing results of the SKPIK method, at the end of this example, we  test  higher dimensional problems, i.e., $n=49408$, $n=197120$ and $m_T=800, 1600, 3200$, which roughly resembles up to a total of 78 million degrees of freedom.  The numerical results are listed in Table \ref{table:ex1-17} and Table \ref{table:ex1-18}, respectively.

\begin{table}
	\centering\caption{Numerical results of the  SKPIK  method for Example~\ref{eddy-ex:1} (2D).} \label{table:ex1-17}
	\begin{tabular}{l |l l l l |l l l l |l l l l}
		\toprule
		$n=49408$   & \multicolumn{4}{c}{$m_T$=800} &  \multicolumn{4}{c}{$m_T$=1600} &  \multicolumn{4}{c}{$m_T$=3200} \\
		$\sigma, \beta$  & r & IT & CPU & RES & r & IT & CPU & RES & r & IT & CPU & RES \\
		\midrule
$\sigma=10^{-4}$ \\
$10^{-2}$  & 6 &  49 & 2.68& 8.67e-7	& 6 & 51 & 2.92& 1.63e-7& 6 & 51 & 3.93& 1.63e-7  \\
$10^{-4}$  &6 &  93 & 7.66& 5.46e-7 &6 &  93 & 10.09& 5.46e-7 & 6 &  93 & 14.57& 5.46e-7  \\
$10^{-6}$  & 4 &  113 & 13.36& 9.59e-7 & 4 &  113 & 20.78& 9.59e-7 & 4 &  113 & 21.46&9.59e-7  \\
$10^{-8}$  & 4 &  51 & 2.71& 6.91e-7 & 4 &  51 & 2.92& 6.91e-7 & 4 &  51 & 3.93& 6.91e-7  \\
$\sigma=1$ \\
$10^{-2}$  & 6 &  87 & 7.16& 8.29e-7& 6 & 93 & 11.21& 5.08e-7& 6 & 93 & 14.59& 6.96e-7  \\
$10^{-4}$  &6 &  93 & 7.66& 6.87e-7 &6 &  93 & 11.26& 8.41e-7 & 6 &  93 & 14.47& 9.94e-7  \\
$10^{-6}$  & 6 &  113 & 13.39& 9.56e-7 & 6 &  113 & 20.06& 9.56e-7 & 6 &  113 & 21.47& 9.56e-7  \\
$10^{-8}$  & 6 &  51 & 2.71& 6.91e-7 & 6 &  51 & 2.92& 6.91e-7 & 6 &  51 & 3.67& 6.91e-7  \\
$\sigma=10^{4}$ \\
$10^{-2}$  & 6 &  21 & 0.63& 8.57e-7& 6 & 22 & 0.87& 9.23e-7& 6 & 23 & 1.10 & 7.29e-7  \\
$10^{-4}$  &6 &  22 & 0.64& 5.99e-7 &6 &  23 & 1.00& 6.10e-7 & 6 &  24 & 1.23& 5.38e-7  \\
$10^{-6}$  & 6 &  22 & 0.64& 9.32e-7 & 6 &  23 & 1.01& 9.92e-7 & 6 &  24 & 1.22& 8.40e-7  \\
$10^{-8}$  & 6 &  24 & 0.75& 4.93e-7 & 6 &  24 & 1.09& 8.10e-7 & 6 &   25 & 1.70& 6.54e-7  \\
\bottomrule
	\end{tabular}
\end{table}

\begin{table}
	\centering\caption{Numerical results of the  SKPIK  method for Example~\ref{eddy-ex:1} (2D).} \label{table:ex1-18}
	\begin{tabular}{l |l l l l |l l l l |l l l l}
		\toprule
		$n=197120$   & \multicolumn{4}{c}{$m_T$=800} &  \multicolumn{4}{c}{$m_T$=1600} &  \multicolumn{4}{c}{$m_T$=3200} \\
		$\sigma, \beta$ & r & IT & CPU & RES & r & IT & CPU & RES & r & IT & CPU & RES \\
		\midrule
$\sigma=10^{-4}$ \\
$10^{-2}$  & 6 &  51 & 15.53& 3.12e-7	& 6 & 51 & 17.47& 3.13e-7& 6 & 51 & 22.03& 3.13e-7  \\
$10^{-4}$  &6 &  99 & 66.35& 6.78e-7 &6 &  99 & 68.28& 9.39e-7 & 6 &  99 & 103.79& 9.72e-7  \\
$10^{-6}$  & 4 &  161 & 163.50& 6.29e-7 & 4 &  161 & 170.83& 6.29e-7 & 4 &  161 & 203.03&6.29e-7  \\
$10^{-8}$  & 4 &  99 & 66.01& 7.03e-7 & 4 &  99 & 68.90& 7.03e-7 & 4 &  99 & 76.90& 7.03e-7  \\
$\sigma=1$ \\
$10^{-2}$  & 6 &  111 & 75.39& 5.94e-7& 6 & 117 & 86.98& 5.49e-7& 6 & 117 & 96.46& 9.25e-7  \\
$10^{-4}$  &6 &  111 & 75.41& 6.78e-7 &6 &  117 & 86.42& 5.91e-7 & 6 &  117 & 96.79& 9.72e-7  \\
$10^{-6}$  & 6 &  161 & 163.88& 6.27e-7 & 6 &  161 &170.03& 6.28e-7 & 6 &  161 & 204.79& 6.28e-7  \\
$10^{-8}$  & 6 &  99 & 66.75& 7.03e-7 & 6 &  99 & 69.90& 7.03e-7 & 6 &  99 & 72.39& 7.02e-7  \\
$\sigma=10^{4}$ \\
$10^{-2}$  & 6 &  34 & 9.65& 6.97e-7& 6 & 36 & 10.34& 7.72e-7& 6 & 40 & 12.97 & 8.73e-7  \\
$10^{-4}$  &6 &  34 & 9.01& 7.72e-7 &6 &  36 & 12.16& 9.69e-7 & 6 & 41 & 14.81& 7.45e-7  \\
$10^{-6}$  & 6 &  38 &11.56& 7.21e-7 & 6 &  38 & 12.99& 8.72e-7 & 6 &  41 & 14.73& 9.77e-7  \\
$10^{-8}$  & 6 &  47 & 14.62& 8.23e-7 & 6 &  47 & 17.64& 8.23e-7 & 6 &  47 & 18.83& 8.26e-7  \\
\bottomrule
	\end{tabular}
\end{table}

The results in Tables \ref{table:ex1-17}-\ref{table:ex1-18} once again illustrate that the SKPIK method is very efficient and shows great robustness concerning different problem sizes and problem parameters.

From the numerical results in Example \ref{eddy-ex:1}, we see that the SKPIK method outperforms the other methods especially when the number of time steps is large. Hence, in the following example, we further test all the methods by a 3D problem to show the robustness of the SKPIK  method.

\begin{example}\label{eddy-ex:2}
{	\rm We consider $\Omega=[0,1]^3$ and $\nu=1$. The desired state is given by
\[
y_d(x,t)=
\begin{pmatrix}
0\\
0\\
\sin(\pi x_1)\sin(\pi x_2)\sin(\pi x_3)
\end{pmatrix}.
\]
}
\end{example}

 In this example, we test all the methods by varying the number of time steps as $m_T=800, 1600, 3200$,  the cost parameters $\beta=10^{-2}, 10^{-4}, 10^{-6}, 10^{-8}$ and the conductivity $\sigma=10^{-4}, 1, 10^4$. The number of spatial discretization nodes varies from $n=1854$, $n=13428$ to $n=102024$. The numerical results corresponding to all the methods are listed in Table \ref{table:ex2-19} ($m_T=800$), Table \ref{table:ex2-20} ($m_T=1600$) and Table \ref{table:ex2-21} ($m_T=3200$).

 The numerical results in Tables \ref{table:ex2-19}-\ref{table:ex2-21} show much analogous computational phenomenon to our previous observations. Therefore, we can conclude that our new method is very robust and  effective for solving large-scale systems from an all-at-once approach to the discretized eddy current optimization problem.

\begin{table}
	\centering\caption{$m_T=800$ for Example~\ref{eddy-ex:2} (3D). \label{table:ex2-19}}
	\begin{tabular}{l l |l l l l  |l l l l |l l l l }
		\toprule
		$n$ &  & \multicolumn{4}{c}{n=1854} &  \multicolumn{4}{c}{n=13428} &  \multicolumn{4}{c}{n=102024} \\
		$\sigma, \beta$ & method & r & IT & CPU & RES & r & IT & CPU & RES & r & IT & CPU & RES \\
		\midrule
$\sigma=10^{-4}$&\\
$10^{-2}$ &  SKPIK  &6  &  42 & 0.35& 5.12e-7& 6 & 68 & 2.06& 7.66e-7& 6 & 96 & 52.37& 9.89e-7  \\
&  LRMINRES  & - & - & - & - & -&  - & - & -& - & - & - & -  \\
&  FMINRES  & &  23 & 352.17 & 3.18e-7 & &  - & - & -&  & - & - & -  \\
$10^{-4}$ &  SKPIK  &6 &  30 & 0.17& 6.17e-7 & 6 &  58 & 1.49& 8.48e-7 & 6&  107 & 55.81& 8.30e-7  \\
&  LRMINRES  &6 &  175 & 41.15 & 9.32e-7 &- &  - & - & -& - & - & - & -  \\
&  FMINRES  & & 22 & 340.65 & 9.51e-7 & &  - & - & -&  & - & - & -  \\
$10^{-6}$ &  SKPIK  &4 &  10 & 0.04& 4.13e-7 & 4 &  21 & 0.31& 6.93e-7 & 4 &  43 & 11.80& 8.30e-7  \\
&  LRMINRES  & 6 & 133 & 44.95 & 9.72e-7 & 6&  271 & 1614 & 9.70e-7& - & - & - & -  \\
&  FMINRES  & &  23 & 669.99 & 3.49e-7 & &  - & - & -&  & - & - & -  \\
$10^{-8}$ &  SKPIK  &3 &  4 & 0.02& 4.44e-7 & 4 &  7 & 0.12& 1.73e-7 & 4 &  13 & 2.55& 5.99e-7  \\
&  LRMINRES  & 6 & 27 & 7.81 & 6.66e-7 & 6&  49 & 326.95 & 6.30e-7& 6 & 59 & 22639 & 9.65e-7  \\
&  FMINRES  & &  15 & 461.66 & 9.87e-8 & &  - & - & -&  & - & - & -  \\
\\
$\sigma=1$\\
$10^{-2}$ &  SKPIK  & 6 &  42 & 0.35& 5.10e-7& 6 & 68 & 2.06& 9.18e-7& 6 & 104 & 53.88& 9.92e-7 \\
&  LRMINRES  & - & - & - & - & -&  - & - & -& - & - & - & -  \\
&  FMINRES  & &  7 & 118.37 & 2.27e-7 & &  - & - & -&  & - & - & -  \\
$10^{-4}$ &  SKPIK  &6 &  30 & 0.17& 5.98e-7 & 6&  58 & 1.42& 8.24e-7 & 6&  107 & 55.64& 8.09e-7   \\
&  LRMINRES  & - & - & - & - & -&  - & - & -& - & - & - & -  \\
&  FMINRES  & &  12.59 & 203.14 & 1.32e-7 & &  - & - & -&  & - & - & -  \\
$10^{-6}$ &  SKPIK  & 4&  10 & 0.04& 4.12e-7&6 & 21 & 0.31 & 6.91e-7 & 6 &  43 & 11.98 & 8.27e-7  \\
&  LRMINRES  & 6 & 243 & 100.67 & 9.87e-7 & 6&  335 & 2416 & 9.83e-7& - & - & - & -  \\
&  FMINRES  & &  9 & 279.38 & 4.28e-7 & &  - & - & -&  & - & - & -  \\
$10^{-8}$ &  SKPIK  &4 &  4 & 0.02& 4.43e-7 & 6 &  7 & 0.11& 1.73e-7 & 6 &  13 & 3.55 & 5.98e-7  \\
&  LRMINRES  & 6 & 29 & 8.41 & 5.85e-7 & 6&  63 & 469.00 & 9.38e-7& 6 & 67 & 27664 & 8.61e-7  \\
&  FMINRES  & &  11 & 358.81 & 1.28e-7 & &  - & - & -&  & - & - & -  \\
\\
$\sigma=10^{4}$\\
$10^{-2}$ &  SKPIK  & 4 &  4 & 0.02& 2.38e-8& 5 & 5 & 0.08& 9.51e-8& 6 & 7 & 2.08& 5.95e-7  \\
&  LRMINRES  & - & - & - & - & -&  - & - & -& - & - & - & -  \\
&  FMINRES  & &  2 & 83.79 & 2.25e-11 & &  - & - & -&  & - & - & -  \\
$10^{-4}$ &  SKPIK  & 4&  4 & 0.02& 2.79e-8 & 5 &  5 & 0.08 & 2.46e-7 & 6 &  7 & 2.06& 8.52e-7  \\
&  LRMINRES  & - & - & - & - & -&  - & - & -& - & - & - & -  \\
&  FMINRES  & &  2 & 88.11 & 2.17e-8 & &  - & - & -&  & - & - & -  \\
$10^{-6}$ &  SKPIK  & 4&  4 & 0.02& 9.13e-8 &5 &  5 & 0.08 & 4.28e-7 &6  &  8 & 2.11 & 1.68e-7  \\
&  LRMINRES  & - & - & - & - & -&  - & - & -& - & - & - & -  \\
&  FMINRES  & &  2 & 87.57 &  5.86e-8 & &  - & - & -&  & - & - & -  \\
$10^{-8}$ &  SKPIK  & 4& 4 & 0.02& 1.10e-7 & 5&  5 & 0.08& 8.60e-7 & 6 &  8 & 2.10 & 3.61e-7  \\
&  LRMINRES  & - & - & - & - & -&  - & - & -& - & - & - & -  \\
&  FMINRES  & &  3 & 115.24 & 3.77e-8 & &  - & - & -&  & - & - & -  \\
\bottomrule
	\end{tabular}
\end{table}

\begin{table}
	\centering\caption{$m_T=1600$ for Example~\ref{eddy-ex:2} (3D). \label{table:ex2-20}}
	\begin{tabular}{l l |l l l l |l l l l |l l l l }
		\toprule
		$n$ &  & \multicolumn{4}{c}{n=1854} &  \multicolumn{4}{c}{n=13428} &  \multicolumn{4}{c}{n=102024} \\
		$\sigma, \beta$ & method & r & IT & CPU & RES & r & IT & CPU & RES & r & IT & CPU & RES \\
		\midrule
$\sigma=10^{-4}$\\
$10^{-2}$ &  SKPIK  & 6 &  42 & 0.45& 5.12e-7& 6 & 68 & 2.25& 7.66e-7& 6 & 96 & 54.99& 9.90e-7  \\
&  LRMINRES  & - & - & - & - & -&  - & - & -& - & - & - & -  \\
&  FMINRES  & &  23 & 773.37 & 3.18e-8 & &  - & - & -&  & - & - & -  \\
$10^{-4}$ &  SKPIK  &6 &  30 & 0.24& 6.17e-7 & 6 &  58 & 1.68& 8.48e-7 & 6&  107 & 58.82& 8.03e-7  \\
&  LRMINRES  & - & - & - & - & -&  - & - & -& - & - & - & -  \\
&  FMINRES  & &  23 & 792.08 & 1.81e-7 & &  - &- & -&  & - & - & -  \\
$10^{-6}$ &  SKPIK  & 4&  10 & 0.07& 4.13e-7 & 4&  21 & 0.40& 6.93e-7 & 4 &  43 & 11.17& 8.30e-7  \\
&  LRMINRES  & - & - & - & - & -&  - & - & -& - & - & - & -  \\
&  FMINRES  & &  23 & - & 3.46e-7 & &  - & - & -&  & - & - & -  \\
$10^{-8}$ &  SKPIK  &3 & 4 & 0.03& 4.44e-7 & 4 &  7 & 0.12& 1.73e-7 & 4 &  13 & 3.88 & 5.99e-7  \\
&  LRMINRES  & 6 & 105 & 67.74 & 9.55e-7 & 6&  89 & 600.37 & 8.88e-7& - & - & - & -  \\
&  FMINRES  & &  15 & - & 9.80e-8 & &  - & - & -&  & - & - & -  \\
\\
$\sigma=1$\\
$10^{-2}$ &  SKPIK  & 6 &  45 & 0.50& 8.14e-7&  6& 68 & 2.25& 9.27e-7& 6 & 105 & 56.35 & 9.39e-7  \\
&  LRMINRES & - & - & - & - & -&  - & - & -& - & - & - & -  \\
&  FMINRES  & &  7 & 266.45 & 7.22e-9 & &  - & - & -&  & - & - & -  \\
$10^{-4}$ &  SKPIK  & 6&  30 &0.24& 5.99e-7 & 6&  58 & 1.66& 8.24e-7 & 6 &  107 & 58.52& 8.09e-7  \\
&  LRMINRES  & - & - & - & - & -&  - & - & -& - & - & - & -  \\
&  FMINRES  & &  10.10 & 374.82 & 1.71e-7 & &  - & - & -&  & - & - & -  \\
$10^{-6}$ &  SKPIK  &6 &  10 & 0.07& 4.12e-7 & 6 &  21 & 0.40& 6.91e-7 & 6 &  43 & 11.17& 8.28e-7  \\
&  LRMINRES  & - & - & - & - & -&  - & - & -& - & - & - & -  \\
&  FMINRES  & &  11 & 698.08 & 1.98e-7 & &  - & - & -&  & - & - & -  \\
$10^{-8}$ &  SKPIK  &4 &  4 & 0.03& 4.43e-7 & 6 & 7 & 0.12 & 1.73e-7 & 6 &  13 & 3.88 & 5.98e-7  \\
&  LRMINRES  & 6 & 492 & 243.56 & 9.83e-7 & -&  - & - & -& - & - & - & -  \\
&  FMINRES  & &  10 &  705.80 & 7.47e-7 & &  - & - & -&  & - & - & -  \\
\\
$\sigma=10^{4}$\\
$10^{-2}$ &  SKPIK  & 4 &  4 & 0.04& 3.34e-8& 5 & 5 & 0.09& 1.32e-7& 6 & 7 & 1.99 & 8.62e-7  \\
&  LRMINRES  & - & - & - & - & -&  - & - & -& - & - & - & -  \\
&  FMINRES  & &  2 & 170.62 & 5.51e-8 & &  - & - & -&  & - & - & -  \\
$10^{-4}$ &  SKPIK  &4 &  4 & 0.04& 3.93e-8 & 5&  5 & 0.09& 3.49e-7 & 6 &  8 & 2.53& 1.36e-7  \\
&  LRMINRES  & - & - & - & - & -&  - & - & -& - & - & - & -  \\
&  FMINRES  & &  2 & 174.64 & 2.00e-8 & &  - & - & -&  & - & - & -  \\
$10^{-6}$ &  SKPIK  &4 &  4 & 0.03& 1.31e-7 &5 &  5 & 0.09& 6.12e-7 & 6 &  8 & 2.53 & 2.45e-7  \\
&  LRMINRES  & - & - & - & - & -&  - & - & -& - & - & - & -  \\
&  FMINRES  & &  2 & 176.18 & 2.61e-8 & &  - & - & -&  & - & - & -  \\
$10^{-8}$ &  SKPIK  &4 &  4 & 0.03& 1.57e-7 & 6 &  6 &0.11& 3.07e-8 & 6 &  8  & 2.53 & 5.21e-7  \\
&  LRMINRES  & - & - & - & - & -&  - & - & -& - & - & - & -  \\
&  FMINRES  & &  2 & 196.44 & 4.44e-7 & &  - & - & -&  & - & - & -  \\
\bottomrule
	\end{tabular}
\end{table}

\begin{table}
	\centering\caption{$m_T=3200$ for Example~\ref{eddy-ex:2} (3D). \label{table:ex2-21}}
	\begin{tabular}{l l |l l l l |l l l l |l l l l }
		\toprule
		$n$ &  & \multicolumn{4}{c}{n=1854} &  \multicolumn{4}{c}{n=13428} &  \multicolumn{4}{c}{n=102024} \\
		$\sigma, \beta$ & method & r & IT & CPU & RES & r & IT & CPU & RES & r & IT & CPU & RES \\
		\midrule
$\sigma=10^{-4}$ \\
$10^{-2}$ &  SKPIK  &6  &  42 & 0.83& 5.12e-7& 6  & 68 & 2.94& 7.66e-7&6  & 97 & 57.26& 9.34e-7 \\
&  LRMINRES  & - & - & - & - & -&  - & - & -& - & - & - & -  \\
&  FMINRES  & &  - & - & - & &  - & - & -&  & - & - & -  \\
$10^{-4}$ &  SKPIK  & 6&  30 & 0.46& 6.17e-7 & 6 &  58 & 2.63 & 8.48e-7 & 6 &  107 & 63.37& 8.30e-7  \\
&  LRMINRES  & - & - & - & - & -&  - & - & -& - & - & - & -  \\
&  FMINRES  & &  - & - & - & &  - & - & -&  & - & - & -  \\
$10^{-6}$ &  SKPIK  &4 &  10 & 0.13& 4.13e-7 & 4 &  21 & 0.54& 4.95e-7 &4 &  43 & 14.37& 8.30e-7 \\
&  LRMINRES  & - & - & - & - & -&  - & - & -& - & - & - & -  \\
&  FMINRES  & &  - & - & - & &  - & - & -&  & - & - & -  \\
$10^{-8}$ &  SKPIK  & 3 &  4 & 0.05 & 4.44e-7 & 4& 7 & 0.16& 1.73e-7 & 4&  13 & 4.13& 5.99e-7  \\
&  LRMINRES  & 6 & 69 & 36.95 & 8.79e-7 & 6&  127 & 852.71 & 9.52e-7& - & - & - & -  \\
&  FMINRES  & &  - & - & - & &  - & - & -&  & - & - & -  \\
\\
$\sigma=1$ \\
$10^{-2}$ &  SKPIK  & 6 &  56 & 1.24& 8.44e-7& 6 & 68 &  2.94& 9.32e-7& 6 & 106 & 62.52& 6.25e-7  \\
&  LRMINRES  & - & - & - & - & -&  - & - & -& - & - & - & -  \\
&  FMINRES  & &  5 & 407.86 & 4.90e-7 & &  - & - & -&  & - & - & -  \\
$10^{-4}$ &  SKPIK  & 6&  31 & 0.48& 6.82e-7 & 6 &  58 & 2.63 &  8.24e-7 & 6&  107 & 63.35 & 8.10e-7  \\
&  LRMINRES  & - & - & - & - & -&  - & - & -& - & - & - & -  \\
&  FMINRES  & &  7.84 & 625.07 & 2.45e-7 & &  - & - & -&  & - & - & -  \\
$10^{-6}$ &  SKPIK  &6 &  10 & 0.13& 4.17e-7 & 6&  21 & 0.54& 6.91e-7 &6 &  43 & 18.43 & 8.28e-7  \\
&  LRMINRES & - & - & - & - & -&  - & - & -& - & - & - & -  \\
&  FMINRES  & &  - & - & - & &  - & - & -&  & - & - & -  \\
$10^{-8}$ &  SKPIK  & 5 &  4 & 0.05 & 4.43e-7 & 6&  7 & 0.16& 1.73e-7 & 6 &  13 & 4.13 & 5.98e-7  \\
&  LRMINRES  & 6 & 75 & 42.05 & 9.49e-7 & 6&  177 & 1205.91 & 9.40e-7& - & - & - & -  \\
&  FMINRES  & &  - & - & - & &  - & - & -&  & - & - & -  \\
\\
$\sigma=10^{4}$ \\
$10^{-2}$ &  SKPIK  & 4 &  4 &0.05& 4.72e-8& 5 & 5 & 0.13& 1.86e-7& 6 & 8 & 2.47 & 7.24e-8  \\
&  LRMINRES  & - & - & - & - & -&  - & - & -& - & - & - & -  \\
&  FMINRES  & &  2 & 370.89 & 1.36e-12 & &  - & - & -&  & - & - & -  \\
$10^{-4}$ &  SKPIK  &4 &  4 & 0.05& 5.55e-8 & 5&  5 & 0.13& 4.95e-7 & 6 &  8  & 2.47 & 1.96e-7  \\
&  LRMINRES  & - & - & - & - & -&  - & - & -& - & - & - & -  \\
&  FMINRES  & &  2 & 350.95 & 1.54e-9 & &  - & - & -&  & - & - & -  \\
$10^{-6}$ &  SKPIK  & 4 &  4 & 0.05 & 1.84e-7 & 5&  5 & 0.13& 8.71e-7 & 6 &  8 & 2.47 & 3.52e-7  \\
&  LRMINRES  & - & - & - & - & -&  - & - & -& - & - & - & -  \\
&  FMINRES  & &  2 & 374.63 & 2.15e-8 & &  - & - & -&  & - & - & -  \\
$10^{-8}$ &  SKPIK  & 4 &  4 & 0.05 & 2.22e-7 &6 &  6 & 0.15 & 4.37e-8 & 6 &  8 & 2.47 & 7.45e-7  \\
&  LRMINRES  & - & - & - & - & -&  - & - & -& - & - & - & -  \\
&  FMINRES  & &  2 & 362.29 & 1.70e-7 & &  - & - & -&  & - & - & -  \\
\bottomrule
	\end{tabular}
\end{table}

\section{Conclusion and Outlook}\label{sec:6}

We construct a splitting-based KPIK method to solve  the discretization from  the eddy current optimal control problems in an all-at-once-based approach.  The SKPIK method relies on the reformulation of a special splitting of the coefficient matrix from the order-reduced KKT system into a matrix equation and the KPIK algorithm for solving the matrix equation. Some properties of the low-rank solution are proposed. Besides, the SKPIK method has been illustrated to be robust concerning different spatial and temporal discretizations and parameters by experiments.

However, the LRMINRES method in our experiments performs poorly because of the bad Schur-complement approximation in the preconditioner. Hence,  improving the Schur-complement approximation would be one of the future work. Another future work will focus on the low-rank approximation method for solving all-at-once discretized eddy current optimal control problems with much more complicated boundary conditions.

\section*{Acknowledgments}
This work was supported by the National Natural Science Foundation of China (Nos. 12126344, 12126337 and 11901324) and the China Scholarship Council (No. 202308350044). 


\section*{Conflict of Interests}

The authors declare that there is no conflict of interests regarding the
publication of this article.
%
%

\bibliographystyle{abbrv}
\bibliography{opt_references}

%
%
%
%

\end{document}